\documentclass[reqno, oneside]{amsart}

\usepackage{chemarr}
\usepackage{hyperref}
\usepackage{comment}
\usepackage{tikz}
\usetikzlibrary{matrix,arrows,decorations}
\usepackage{tikz-cd}
\usepackage{cleveref}
\usepackage{adjustbox}
\usepackage[textsize=scriptsize]{todonotes}

\usepackage[a4paper]{geometry}
\theoremstyle{definition}
\newtheorem{nul}{}[section]

\newtheorem{rmk}[nul]{Remark}

\newtheorem{cnstr}[nul]{Construction}

\newtheorem{exm}[nul]{Example}

\newtheorem{rec}[nul]{Recollection}

\newtheorem*{dfn*}{Definition}
\newtheorem*{axm*}{Axiom}
\newtheorem*{ntn*}{Notation}
\newtheorem*{exm*}{Example}
\newtheorem*{exr*}{Exercise}
\newtheorem*{int*}{Intuition}
\newtheorem*{qst*}{Question}
\newtheorem*{rmk*}{Remark}

\newtheorem{thm}[nul]{Theorem}
\newtheorem{thmx}{Theorem}
\newtheorem*{thm:main}{Theorem \ref{thm:main}}
\newtheorem*{thm:iteratedK}{Theorem \ref{thm:iteratedK}}

\newtheorem{prop}[nul]{Proposition}
\newtheorem{lem}[nul]{Lemma}

\newtheorem{cor}{Corollary}[nul]
\newtheorem*{thm*}{Theorem}
\newtheorem*{prop*}{Proposition}
\newtheorem*{cor*}{Corollary}
\newtheorem*{lem*}{Lemma}
\newtheorem*{cnj*}{Conjecture}

\DeclareMathOperator*{\colim}{\mathrm{colim}}

\DeclareMathOperator{\Hom}{\text{Hom}}

\DeclareMathOperator{\bD}{\mathbf{D}}

\newcommand{\Z}{\mathbb{Z}}
\DeclareMathOperator{\E}{\mathbb{E}}

\DeclareMathOperator{\Q}{\mathbb{Q}}

\DeclareMathOperator{\p}{\mathfrak{p}}
\DeclareMathOperator{\G}{\mathbb{G}}
\DeclareMathOperator{\F}{\mathbb{F}}

\DeclareMathOperator{\Spec}{\mathrm{Spec}}
\DeclareMathOperator{\Spf}{\mathrm{Spf}}

\DeclareMathOperator{\CAlg}{\mathrm{CAlg}}

\DeclareMathOperator{\Sp}{\mathrm{Sp}}

\DeclareMathOperator{\THH}{\mathrm{THH}}
\DeclareMathOperator{\sh}{\mathrm{sh}}
\DeclareMathOperator{\Perf}{\mathrm{Perf}}
\newcommand{\tp}{t_+}
\newcommand{\Alt}[2]{\mathrm{Alt}^{(#1)}_{#2}}

\begin{document}

\title{Examples of chromatic redshift in algebraic $K$-theory}
\author{Allen Yuan}

\maketitle

\begin{abstract}
We give a simple argument to detect chromatic redshift in the algebraic $K$-theory of $\E_{\infty}$-ring spectra and give two applications: we show for $n\geq 1$ that $K(E_n)$, the algebraic $K$-theory of any height $n$ Lubin-Tate theory, has nontrivial $T(n+1)$-localization, and that $K^{(n)}(k)$, the $n$-fold iterated algebraic $K$-theory of a field $k$ of characteristic different from $p$, has nontrivial $T(n)$-localization.  \end{abstract}
\setcounter{tocdepth}{1}
\tableofcontents{}

\section{Introduction}

Chromatic homotopy theory organizes generalized cohomology theories on spaces by \emph{chromatic height}.  Roughly, the chromatic height of a theory measures the complexity of the torsion seen by the theory.  Height 0 theories, such as rational cohomology, see no torsion information.  At height 1, one sees more sophisticated cohomology theories such as real or complex $K$-theory; nevertheless, these theories remain relatively well understood.  At heights 2 and above, the situation rapidly becomes more complicated: one has elliptic cohomology and topological modular forms at height 2, and Lubin-Tate theories $E_n$ and topological automorphic forms at higher height, but these theories are computationally more difficult and elude concrete geometric understanding \cite{LRS, GH, Ell2, TAF}.

This paper is concerned with a phenomenon in which cohomology theories of higher height arise from theories of lower height: the chromatic redshift in algebraic $K$-theory.  Motivated by finding ring spectrum analogues of the Lichtenbaum-Quillen conjectures, Ausoni and Rognes computed the (mod $(p,v_1)$) algebraic $K$-theory of (the Adams summand of) connective $K$-theory \cite{AusRog}.  From these computations, they observed that algebraic $K$-theory tends to increase chromatic complexity by one: that is, the algebraic $K$-theory of a height $n$ ring spectrum is often of height $n+1$.  This became known as the \emph{chromatic redshift philosophy}.

Chromatic redshift has since been studied extensively: in height 1 examples such as $ku^{\wedge}_p, KU^{\wedge}_p$ and $ku/p$ by \cite{BlumMan, Ausoni, AusRog, AusRog2}; computationally in \cite{Angelini-Knoll, AKQ, Veen} and via theoretical considerations \cite{Westerland, CSY}; additionally, questions of descent were addressed in the foundational papers \cite{CMNNausrog, CM}.  
This work builds upon recent progress by Land-Mathew-Meier-Tamme \cite{LMMT} and 
Clausen-Mathew-Naumann-Noel \cite{CMNN}: to motivate it, we sample a result of \cite{CMNN} which generalizes a theorem of Mitchell \cite{Mitchell} in the case $n=1$:

\begin{thm}[Theorem A, \cite{CMNN}]\label{thm:introCMNN}
Let $R$ be an $\E_{\infty}$-ring spectrum and let $n\geq 0$.  If $L_{T(n)}R = 0,$ then $L_{T(n+1)}K(R) = 0$. 
\end{thm}

\begin{rmk}
Here, $L_{T(n)}$ denotes Bousfield localization at a height $n$ telescope $T(n)$; the non-vanishing of this localization is a sense in which a spectrum is ``supported at height $n$."  We will say an $\E_{\infty}$-ring $R$ is \emph{of height} $n\geq 0$ if $L_{T(n)}R \not\simeq 0$ and $L_{T(n+1)} R \simeq  0$; in fact, in this situation, a theorem of Hahn \cite{Hahn} implies that $L_{T(j)} R \simeq 0 $ for all $j>n$.
\end{rmk}  

Roughly speaking, \Cref{thm:introCMNN} asserts that the algebraic $K$-theory of a height $n$ $\E_{\infty}$-ring is of height \emph{at most} $n+1$.  Following this result, the question remained whether height shifting actually occurs; that is, when is the algebraic $K$-theory of a height $n$ ring is of height \emph{exactly} $n+1$?   In this paper, we give a simple and non-computational explanation for this height shifting in a range of examples. We show:

\begin{thmx} \label{thm:Eredshift}
Let $k$ be a perfect field of characteristic $p$, let $\mathbb{G}_0$ be a $1$-dimensional formal group over $k$ of height $n\geq 1$, and let $E_{n} = E_{k,\mathbb{G}_0}$ denote the associated Lubin-Tate theory.  Then $L_{T(n+1)} K(E_n)$ is nonzero.  
\end{thmx}

\begin{thmx}\label{thm:iteratedK}
Let $n\geq 0$ and let $K^{(n)}$ denote the $n$-fold iterate of algebraic $K$-theory.  Then, for any field $k$ of characteristic different from $p$, the spectrum $L_{T(n)}K^{(n)}(k)$ is nonzero.  
\end{thmx}

We remark that the question of height shifting has recently been independently addressed by Hahn and Wilson \cite{HahnWilson} in the case of truncated Brown-Peterson spectra $BP\langle n\rangle$; their approach is more computational in nature, and further proves Lichtenbaum-Quillen-type statements in this case.  The nonvanishing of $L_{T(n+1)}K(BP\langle n\rangle)$ is closely related to \Cref{thm:Eredshift}, though we do not know if either result implies the other.  
\Cref{thm:Eredshift} was previously known in the case $n=1, p\geq 5$ by \cite{Ausoni, AusRog, BlumMan}, and the cases $p=2,3$ also follow from \cite{HahnWilson}.

\Cref{thm:iteratedK} addresses another case of interest in Rognes's redshift conjectures, that of iterated $K$-theory (of fields), in a form stated by Barwick \cite{BarMultAKT}.  There has been previous work on this case by \cite{Veen, Angelini-Knoll, LSW, BCD}, and the theorem answers the $K$-theory variant of a question of Hesselholt-Madsen about the chromatic complexity of $\mathrm{TC}^{(n)}(\F_p)$ \cite[\S 4]{HM}.  

 The $n$-fold iterated $K$-theory of a field $k$ is related to an $(n-1)$-fold categorification of the category of $k$-vector spaces (cf. \cite{RognesMSRI}) -- for instance, when $n=2$, the work of Baas, Dundas, Richter and Rognes expresses $K(K(k))$ as a classifying space for certain $k$-linear categories \cite[Theorem 1.1]{BDRR}.  Moreover, motivated by the work of Ausoni-Rognes \cite{AusRog}, Baas-Dundas-Rognes describe the cohomology theory $K(ku)$ (which is $T(2)$-locally equivalent to $K(K(\mathbb{C}))$) geometrically via virtual $2$-vector bundles \cite{BDR}.   \Cref{thm:iteratedK} asserts that this relatively tangible and geometric categorification procedure produces cohomology theories of arbitrary height.  In addition, we remark that the proof of \Cref{thm:iteratedK} is at least in part ``effective,'' as it describes a ring map from $K^{(n)}(\Q)$ to rings constructed out of arbitrary height Lubin-Tate theories; in particular, at height $2$, one obtains $\E_{\infty}$-ring maps from $K(K(\Q))$ to certain elliptic cohomology theories. 

\begin{rmk}\label{rmk:dhsnilp}
The results in this paper hold equally well with $L_{T(n)}$ replaced by the Bousfield localization functors $L_{K(n)}$ at a Morava $K$-theory at height $n$.  The distinction between $T(n)$ and $K(n)$ is invisible to our discussion because a \emph{ring} spectrum is $T(n)$-acyclic if and only if it is $K(n)$-acyclic, essentially by the Devinatz-Hopkins-Smith nilpotence theorem (cf. \cite[Lemma 2.3]{LMMT},  \cite[\S 4.4]{CSYteleambi}).
\end{rmk}

\subsection{Outline of methods}

Suppose $E$ is a $T(n)$-local $\E_{\infty}$-ring.  Then one can consider the $\E_{\infty}$-ring $E^{tC_p}$, the Tate cohomology of $C_p$ acting trivially on $E$.  A theorem of Kuhn (cf. \Cref{sec:unblueshift}) implies that $L_{T(n)}E^{tC_p} \simeq 0$, so $E^{tC_p}$ is an $\E_{\infty}$-ring of height at most $n-1$ (a phenomenon dubbed ``chromatic blueshift'').  The foundation for our examples is the following result, which asserts that \emph{chromatic redshift} occurs for rings of the form $E^{tC_p}$:

\begin{thmx} \label{thm:main}
Let $E$ be an $\E_{\infty}$-ring spectrum such that $L_{T(n)} E $ is nonzero.  Then $L_{T(n)}K(E^{tC_p})$ is nonzero.
\end{thmx}

The proof uses the following elementary observation:

\begin{rmk}\label{rmk:ringmap}
Note that the zero ring admits no nonzero modules.  Thus, if $A \to B$ is a map of ring spectra and $B$ is nonzero, then $A$ is nonzero.  In particular, in order to show a ring $R$ is $T(n)$-locally nonzero, it suffices to produce a ring map from $R$ to another ring which is $T(n)$-locally nonzero.  We will apply this observation throughout this paper.  
\end{rmk}

We now outline the proof of \Cref{thm:main}, with the full proof to be given in \Cref{sec:proofmain}:


\begin{enumerate}
\item First, we reduce to the case when $E$ is connective. \label{red1}
\item The Dennis trace provides a map of rings $K(E^{tC_p}) \to \THH(E^{tC_p})^{hS^1}.$  
\item Using the universal property of $\THH$ in $\E_{\infty}$-rings, we produce a map of $\E_{\infty}$-rings $$\THH(E^{tC_p})^{hS^1} \to (E^{tC_p})^{hS^1/C_p}.$$  Here, the action of $S^1/C_p$ on $E^{tC_p}$ is the residual action, arising because the trivial action of $C_p$ on $E$ can be extended to a trivial action of $S^1$.  
\item Since $E$ is connective by (\ref{red1}), the ``Tate orbit lemma'' of Nikolaus-Scholze \cite{NS} provides a $p$-complete equivalence $(E^{tC_p})^{hS^1/C_p} \simeq_p E^{tS^1}$.  We will see that because $L_{T(n)}E \neq 0$ by assumption, it follows that $L_{T(n)}E^{tS^1} \neq 0$ (for instance, when $E$ is complex oriented, $E$ is a summand of $E^{tS^1}$).
This step is where one can ``see the redshift happen," and it will be explained in detail in \Cref{sec:unblueshift}.  
\item Combining the above steps and applying \Cref{rmk:ringmap}, we conclude that $K(E^{tC_p})$ is $T(n)$-locally nontrivial, since it admits a ring map to a $T(n)$-locally nontrivial ring.  
\end{enumerate}

Applying \Cref{rmk:ringmap}, \Cref{thm:main} can be leveraged to deduce chromatic redshift for rings mapping to $E^{tC_p}$ for appropriately chosen $E$. 

For the case of Lubin-Tate theories (\Cref{thm:Eredshift}), the idea is to attempt to construct a ring map from $E_n \to E_{n+1}^{tC_p}$ for a height $n+1$ Lubin-Tate theory $E_{n+1}$.  It turns out that this can be done ``up to \'{e}tale extension," which then suffices by applying the descent results of Clausen-Mathew \cite{CM}; we thank Akhil Mathew for explaining this point to us.  We give the details in \Cref{sec:proofEredshift}.

We prove \Cref{thm:iteratedK} in \Cref{sec:iteratedK}.  For the case $k=\Q$ (and in fact, $k=\Q(\zeta_p)$), one roughly uses the $n$-fold iterate of the proof of \Cref{thm:main}.  The proof then uses an argument of Akhil Mathew to extend to the general case: we are grateful to him for explaining this to us and allowing us to reproduce his arguments, which appear in \Cref{sub:generalk}.  


\subsection{Acknowledgments} 

The author would like to thank Andrew Blumberg, Robert Burklund, Shachar Carmeli, Jeremy Hahn, Jacob Lurie, Akhil Mathew, and Arpon Raksit for numerous helpful suggestions and conversations related to this material.  Special thanks are due to Akhil Mathew for generously contributing the ideas in \Cref{lem:stalk} and \Cref{sub:generalk} and allowing us to use his arguments in this paper, and to Jeremy Hahn and Arpon Raksit for asking about the case of iterated $K$-theory.  The author was supported in part by NSF grant DMS-2002029.

\section{Recollections on the Tate orbit lemma}\label{sec:unblueshift}
Roughly, our approach to demonstrating chromatic redshift is by showing that any blueshift (i.e., lowering of chromatic height) arising from Tate cohomology is reversed by algebraic $K$-theory.  The purpose of this (entirely expository) section is to explain the mechanism for this ``un-blueshift.''  The key idea is the Tate orbit lemma of Nikolaus-Scholze \cite{NS}, which we review in \Cref{sub:tateorbit}.

\subsection{Tate cohomology and blueshift}
Let $X$ be a spectrum with the action of a finite group $G$.  Then one can consider $X^{tG}$, the Greenlees-May Tate cohomology of $G$ with coefficients in $X$ \cite{GreenleesMay} , which is defined via the cofiber sequence:
\[
X_{hG} \xrightarrow{\mathrm{tr}} X^{hG} \to X^{tG}.
\]
Here, the map $\mathrm{tr}$ is the additive transfer, which can be thought of as taking an orbit $\overline{x}$ to its sum $\sum_g gx$.  For us, the group $G$ will most often be $C_p$, the cyclic group of $p$ elements, and unless otherwise stated, the operation $(-)^{tG}$ will be taken with respect to the trivial action.  We remark also that $(-)^{tG}$ is lax symmetric monoidal \cite[Theorem I.3.1]{NS}, and thus takes $\E_{\infty}$-rings to $\E_{\infty}$-rings.  

From work of  Greenlees, Hovey, Sadofsky, and Kuhn \cite{GS96, HS96, Kuhn}, it became understood that Tate cohomology exhibited a ``blueshift" phenomenon: it tends to shift chromatic height down.  More precisely, Kuhn proved the following theorem:

\begin{thm}[Kuhn \cite{Kuhn}]\label{thm:Kuhn}
Let $n\geq 0$ and let $X$ be a $T(n)$-local spectrum with an action of $C_p$.  Then $L_{T(n)}X^{tC_p} \simeq 0$.  
\end{thm}

\begin{rmk}\label{rmk:cxorientedtCp}
Suppose that $E$ is a complex oriented ring spectrum such that $[p](x)$, the $p$-series of its associated formal group, is not a zero divisor in $E_*[[x]]$.  Then, one can compute the homotopy groups of $E^{tC_p}$ as:
\[
\pi_*E^{tC_p} \cong E_* ((x))/[p](x), \quad |x| = -2.
\]
\end{rmk}

This description of $E^{tC_p}$ can be used to understand Kuhn's theorem more concretely in some cases, as the following example illustrates:

\begin{exm}\label{exm:kutcp}
If $E$ is the $p$-complete complex $K$-theory spectrum $KU_p$, then while $KU_p$ is $T(1)$-local, we have
\[
\pi_* KU_p^{tC_p} \simeq \frac{\Z_p[\beta^{\pm 1}]((x))}{\beta^{-1}((1+\beta x)^p - 1)} \simeq \Q_p(\zeta_p)((x)),
\]
so $KU_p^{tC_p}$ is rational (the key point being that, using the relation, inverting $x$ also inverts $p$).  The last identification comes from noting that $1+\beta x$ becomes a primitive $p$th root of unity $\zeta_p$.  
\end{exm}

\subsection{The Tate orbit lemma and reversing blueshift}\label{sub:tateorbit}

Let $X$ be a spectrum.  We saw in the previous section that we may form $X^{tC_p}$, where $X$ is given the trivial $C_p$ action.  However, the trivial action of $C_p$ extends to a trivial action of the circle $S^1$: it follows that we can regard $X$ as having trivial $S^1$-action and regard $X^{tC_p}$ as the Tate cohomology with respect to the subgroup $C_p \subset S^1$.  As such, $X^{tC_p}$ acquires a residual action of the quotient $S^1/C_p$ (which is isomorphic to $S^1$).  

This residual action of $S^1/C_p$ is critical to our approach: we will see that in the case when $X$ is a connective ring spectrum, while $X^{tC_p}$ can be of lower height than $X$, the homotopy fixed points $(X^{tC_p})^{hS^1/C_p}$ are always of the same height as $X$ (cf. \Cref{rmk:unblueshift}).  We first illustrate this with an example:

\begin{exm}\label{exm:Zcase}
Consider the case $X = \Z$; then $$\pi_*\Z^{tC_p} = \F_p((x)), \quad |x| = -2.$$  It will be instructive to think of the ring $\Z^{tC_p}$ as having height $-1$ because $v_0 := p$ vanishes.  Note that there is a canonical $S^1/C_p$-equivariant map $\Z^{hC_p} \to \Z^{tC_p}$, which induces a map
\begin{equation}\label{eqn:zhcp}
(\Z^{hC_p})^{hC_{p^2}/C_p} \to (\Z^{tC_p})^{hC_{p^2}/C_p}.
\end{equation}
But the source can be identified with $\Z^{hC_{p^2}}$, which has homotopy groups $\Z[[x]]/p^2x$, $|x|=-2$.  Comparing the homotopy fixed point spectral sequences in (\ref{eqn:zhcp}), one deduces that the spectral sequence in the target adds a class $y$ (coming from $H^2(C_{p^2}/C_p;\F_p)$) which satisfies the equation $y=px$ (via an additive extension in the spectral sequence).  It follows that 
\[
\pi_*(\Z^{tC_p})^{hC_{p^2}/C_p} \simeq \Z/p^2 ((x)), \quad |x|= -2.
\]
Thus, while $\Z \mapsto \Z^{tC_p}$ sends $v_0=p$ to zero, taking homotopy fixed points by $C_{p^2}/C_p$ starts to add back the powers of $p$.  In fact, one can do a similar computation for the subgroups $C_{p^k}/C_p \subset S^1/C_p$ and $S^1/C_p$ itself:
\begin{align*}
\pi_* (\Z^{tC_p})^{hC_{p^k}/C_p} &= \Z/p^k((x)), \quad |x|= -2 \\
\pi_* (\Z^{tC_p})^{hS^1/C_p} &= \Z_p ((x)),  \quad |x|= -2. 
\end{align*}

Hence, taking fixed points for the action of $S^1/C_p$ passes from characteristic $p$ back to characteristic $0$, which we can view as height shifting from height $-1$ to height $0$.  
\end{exm}

In fact, the phenomenon observed in the above example is a special case of the following result of Nikolaus-Scholze:

\begin{prop}[Tate orbit lemma, \cite{NS}, Lemma II.4.2]\label{prop:tateorbit}
Let $X$ be a bounded below spectrum with $S^1$-action.  Then there is a natural map $X^{tS^1} \to (X^{tC_p})^{hS^1/C_p}$ which exhibits the target as the $p$-completion of the source.  
\end{prop}

We refer the reader to \cite[Lemma II.4.2]{NS} for the proof.  The idea is that the statement commutes with limits of Postnikov towers, so one reduces to the case when $X$ is discrete, in which case the statement can be understood directly in the spirit of \Cref{exm:Zcase}.  

Our interest in \Cref{prop:tateorbit} stems from the fact that unlike $(-)^{tC_p}$, the functor $(-)^{tS^1}$ tends to \emph{not} lower chromatic height (as illustrated in \Cref{exm:Zcase}):

\begin{prop}\label{prop:tS1height}
Suppose $E$ is a homotopy commutative ring spectrum such that $L_{T(n)}E \not\simeq 0$.  Then $L_{T(n)}E^{tS^1}\not\simeq 0$.  
\end{prop}
\begin{proof}
Since there is a ring map $E \to E\otimes K(n)$, it suffices to show that $L_{T(n)}(E\otimes K(n))^{tS^1}$ is nonzero.  But by \Cref{rmk:dhsnilp}, we have $L_{T(n)}E \simeq 0$ iff $L_{K(n)} E \simeq 0$ iff $L_{K(n)} (E\otimes K(n)) \simeq 0$ iff $L_{T(n)} (E\otimes K(n)) \simeq 0$.  Thus, we may assume without loss of generality that $E$ is complex oriented by replacing $E$ with $E\otimes K(n)$.

In the special case that $E$ is complex oriented, we have $$\pi_* E^{hS^1} \cong E^*(BS^1) \cong  E_*[[x]],\quad |x|=-2$$ and $E^{tS^1} \simeq E^{hS^1}[x^{-1}]$; it follows that $E^{tS^1}$ has $E$ as an $E$-module summand, which completes the proof.  
\end{proof}

\begin{rmk}\label{rmk:unblueshift}
\Cref{prop:tateorbit} and \Cref{prop:tS1height} are ultimately ``where the redshift happens" in our arguments.  To illustrate this, let $E$ be a $K(n)$-local ring spectrum and let $e\to E$ denote its connective cover.  Then, by Kuhn's theorem (\Cref{thm:Kuhn}) and \Cref{lem:covertate}, we have $$L_{T(n)}e^{tC_p} \simeq L_{T(n)} E^{tC_p} \simeq 0,$$ so $e^{tC_p}$ is of height at most $n-1$.  On the other hand, by \Cref{prop:tateorbit} and \Cref{prop:tS1height}, $(e^{tC_p})^{hS^1/C_p}$ is equivalent to the $p$-completion of $e^{tS^1}$, and therefore has height $n$.  Thus, the operation $(-)^{hS^1/C_p}$ shifts the height of $e^{tC_p}$ up by one.  
\end{rmk}

\begin{rmk}
In fact, passing to the connective cover (to satisfy the bounded below hypothesis in \Cref{prop:tateorbit}) is essential.  We saw in \Cref{exm:kutcp} that $KU_p^{tC_p}$ is rational, and since rational spectra are closed under limits, it follows that $(KU_p^{tC_p})^{hS^1/C_p}$ is also rational (and in particular, not equivalent to $KU_p^{tS^1}$).  On the other hand, if one uses \emph{connective} $K$-theory $ku_p$, then $\pi_0ku_p^{tC_p} \cong \Z_p[\zeta_p]$, which is not rational.  
\end{rmk}

\section{Proof of \Cref{thm:main}}\label{sec:proofmain}

Here, we prove \Cref{thm:main} following the outline in the introduction:

\begin{thm:main}
Let $E$ be an $\E_{\infty}$-ring spectrum such that $L_{T(n)} E $ is nonzero.  Then $L_{T(n)}K(E^{tC_p})$ is nonzero.
\end{thm:main}

\begin{proof}
We begin by reducing to the connective case.  The main ingredient in the reduction is the following landmark result of Land-Mathew-Meier-Tamme \cite{LMMT} and Clausen-Mathew-Noel-Naumann \cite{CMNN}:

\begin{thm}[\cite{LMMT}, Theorem A]\label{thm:LMMT}
Let $n\geq 1$, and let $A\to B$ be a map of $\E_1$-ring spectra which is a $T(n-1)\oplus T(n)$-equivalence.  Then $K(A)\to K(B)$ is a $T(n)$-equivalence.  
\end{thm}

We also note the following lemma, which also appears in \cite[Lemma 4.7]{CMNN}; we reproduce their proof for the reader's convenience: 

\begin{lem}\label{lem:covertate}
Let $X$ be a coconnective spectrum.  Then $X^{tC_p}$ is $T(n)$-acyclic for all $n\geq 0$. 
\end{lem}
\begin{proof}
Because $(-)^{tC_p}$ commutes with filtered colimits on coconnective spectra, one can reduce to the case when $X$ is concentrated in a single degree.  In this case, $X^{tC_p}$ is a $p$-torsion $\mathbb{Z}$-module, and thus $T(n)$-acyclic for any $n\geq 0$.  
\end{proof}

Combining this lemma with \Cref{thm:LMMT}, we obtain:

\begin{cor}\label{cor:conncover}
Let $A$ be an $\E_1$-ring spectrum and $a\to A$ be its connective cover.  Then for any $n\geq 1$, the natural map $K(a^{tC_p}) \to K(A^{tC_p})$ is a $T(n)$-equivalence.  
\end{cor}

Hence, we are free to replace $E$ by its connective cover $e = \tau_{\geq 0}E$.  To finish the proof, we will use $\THH$ and the Dennis trace map, so we first recall the following fact about $\THH$:

\begin{rec}\label{THHuprop}
Let $A$ be an $\E_{\infty}$-ring spectrum and let $B$ be an $\E_{\infty}$-ring spectrum with an action of $S^1$.  Then $\THH(A)$ is given by the colimit over $S^1$ of the constant diagram at $A$ in $\E_{\infty}$-rings.  Hence, restriction along $A\to \THH(A)$ in the diagram
\[
\begin{tikzcd}
A \arrow[d] \arrow[rd] & \\
 \THH(A) \arrow[r, dashed]& B
\end{tikzcd}
\]
induces a bijection between homotopy classes of $S^1$-equivariant $\E_{\infty}$-ring maps $\THH(A) \to B$ and homotopy classes of (non-equivariant) $\E_{\infty}$-ring maps $A\to B$.
\end{rec}

We apply this in the case when $A = B = e^{tC_p}$, and where $B$ is given a \emph{nontrivial} $S^1$-action as follows: we saw in the previous section that $e^{tC_p}$ acquires an (interesting) residual $S^1/C_p$ action, and we give $B = e^{tC_p}$ an $S^1$-action by choosing an isomorphism $S^1 \cong S^1/C_p$.  Then, by the above universal property of $\THH$, the identity map $e^{tC_p} \to e^{tC_p}$ extends canonically to an $S^1$-equivariant $\E_{\infty}$ map $\THH(e^{tC_p}) \to e^{tC_p}$ (for this nontrivial $S^1$ action on the target).  Taking fixed points, this determines an $\E_{\infty}$-ring map $\THH(e^{tC_p})^{hS^1} \to (e^{tC_p})^{hS^1/C_p}$.  Composing this with the (circle invariant) Dennis trace yields a map of $\E_{\infty}$-rings 
\begin{equation*}\label{eqn:comp}
K(e^{tC_p}) \to \THH(e^{tC_p})^{hS^1} \to  (e^{tC_p})^{hS^1/C_p}.
\end{equation*}

Since $e$ is connective, the Tate orbit lemma (\Cref{prop:tateorbit}) identifies the target ring with the $p$-completion of $e^{tS^1}$, which has nontrivial $T(n)$-localization by \Cref{prop:tS1height}.  Hence, $K(e^{tC_p})$ is $T(n)$-locally nontrivial because it admits a ring map to a $T(n)$-locally nontrivial ring.   
\end{proof}

\begin{rmk}
Asaf Horev has pointed out that $A$ is an $S^1$-equivariant module over $\THH(A)$ for any \emph{framed $\E_2$-algebra} $A$, not just for $\E_{\infty}$-rings.  Given this, one can wonder if \Cref{thm:main} holds in some greater generality.  
\end{rmk}

\section{Lubin-Tate theories}\label{sec:proofEredshift}

Let $E_n$ denote the Lubin-Tate theory associated to a formal group $\mathbb{G}_0$ of height $n$ over a perfect field $k$ of characteristic $p$.  The goal of this section will be to show that $K(E_n)$ is an $\E_{\infty}$-ring spectrum of height $n+1$; that is, that $L_{T(n+1)}K(E_n)$ is nonzero.

One way to prove this would be to exhibit an $\E_{\infty}$-ring $A$ such that $L_{T(n+1)}K(A) \not\simeq 0$ together with an $\E_{\infty}$-ring map $E_n \to A$.  
In fact, by work of Goerss-Hopkins-Miller \cite{GH} and Lurie \cite{Ell2}, Lubin-Tate theories are characterized by a universal property which makes them amenable to mapping out of.  In order to state this universal property, we first recall a definition:

\begin{rec}\label{rec:landweberideal}
Let $R$ be a complex orientable ring spectrum.  Then for each $n\geq 1$, there is a canonical \emph{$n$th Landweber ideal} $\mathcal{I}_n \subset \pi_*(R)$, generated by (any choice of) $p, v_1, \cdots , v_{n-1}$.  In the following, we will be working only with $2$-periodic complex orientable rings, and we will denote the corresponding ideal of $\pi_0(R)$ by $I_n\subset \pi_0(R)$.  
\end{rec}

\begin{thm}[Goerss-Hopkins-Miller \cite{GH}, Lurie \cite{Ell2} Theorem 5.0.2]\label{thm:Eunivprop}
Let $E_n$ denote the Lubin-Tate theory associated to a height $n$ formal group $\mathbb{G}_0$ over a perfect field $k$ of characteristic $p$, let $R$ be a $2$-periodic complex orientable $K(n)$-local $\E_{\infty}$-ring spectrum, and let $\mathbb{G}_R = \Spf R^0(\mathbb{C}P^{\infty})$ denote the canonical Quillen formal group over $\pi_0(R)$.  

Then there is a canonical homotopy equivalence between:
\begin{enumerate}
\item The space $\Hom_{\CAlg}(E_n, R)$ of $\E_{\infty}$-ring maps from $E_n$ to $R$.
\item The set of pairs $(f,\alpha)$, where $f: k \to \pi_0(R)/I_n$ is a ring homomorphism, and $\alpha : f^*\mathbb{G}_0 \cong (\mathbb{G}_R)_{\pi_0(R)/I_n}$ is an isomorphism of formal groups over $\pi_0(R)/I_n$.  
\end{enumerate}
\end{thm}

Given a $K(n)$-local $\E_{\infty}$-ring $A$, this theorem provides an algebraic method for producing $\E_{\infty}$-maps from $E_n$ to $A$.  The remaining question is how to choose the $\E_{\infty}$-ring $A$.  By \Cref{thm:main}, an example of an $\E_{\infty}$-ring $A$ which exhibits redshift is $E_{n+1}^{tC_p}$, where $E_{n+1}$ is a Lubin-Tate theory of corresponding to a height $n+1$ formal group over $k$.  Unfortunately, while the formal group of $E_{n+1}^{tC_p}$ does have height $n$ at certain residue fields, these fields are generally finite extensions of Laurent series rings over $k$; in general, the corresponding formal groups are complicated, and not isomorphic to formal groups defined over $k$ (cf. \cite{AMS} for a more detailed discussion).

However, one does not quite need to produce such a map; roughly speaking, we will see that the \'{e}tale descent results of Clausen-Mathew-Naumann-Noel \cite{CMNNausrog} imply that one only needs to produce this map after a sequence of \'{e}tale extensions; this allows one to reduce to the case of a separably closed residue field, where all formal groups of a given height \emph{are} isomorphic.  The key ingredient powering this reduction is the following theorem, which asserts that the functor $L_{T(n)} K(-)$ satisfies \'{e}tale descent:

\begin{thm}[Theorem A.4 \cite{CMNNausrog}]\label{thm:etaledesc}
Let $R$ be an $\E_{\infty}$-ring spectrum and $n\geq 0$.  Then the construction $Y\mapsto L_{T(n)}K(\Perf (Y))$ defines an \'{e}tale sheaf on $R$.  
\end{thm}

The author would like to thank Akhil Mathew for explaining how to apply this theorem, and especially for explaining \Cref{lem:stalk} (of course, any errors in the argument produced here are solely the responsibility of the author).  

\subsection{Proof of redshift for $E_n$}

We now give the details of the argument.  

\begin{cnstr}\label{cnstr:stricthens}
Let $R$ be an $\E_{\infty}$-ring spectrum and $\p \subset \pi_0(R)$ be a prime ideal.   Then one can produce a ring $\pi_0(R)^{\sh}_{\p}$, known as the \emph{strict henselization} of $\pi_0(R)$ at $\p$, with the following properties (see, for instance, \cite[\href{https://stacks.math.columbia.edu/tag/0BSK}{Tag 0BSK}]{stacks-project}):
\begin{enumerate}
\item $\pi_0(R)^{\sh}_{\p}$ is strictly henselian; in particular, it is a local ring with separably closed residue field.  
\item $\pi_0(R)^{\sh}_{\p}$ is a filtered colimit of \'{e}tale $\pi_0(R)$-algebras.
\item $\p \pi_0(R)^{\sh}_{\p}$ is the maximal ideal of $\pi_0(R)^{\sh}_{\p}.$  
\end{enumerate}

By \cite[Theorem 7.5.0.6]{HA}, this construction can be lifted (essentially uniquely) to produce an $\E_{\infty}$-ring $R^{\sh}_{\p}$ which is a filtered colimit of \'{e}tale extensions of $R$ and such that $\pi_0$ of the map $R \to R^{\sh}_{\p}$ is identified with $\pi_0(R) \to \pi_0(R)^{\sh}_{\p}$.
\end{cnstr}

\begin{lem}\label{lem:stalk}
Let $R$ be an $\E_{\infty}$-ring spectrum, and suppose that $L_{T(n)} K(R) \not\simeq 0$ for some integer $n\geq 0$.  Then there exists a prime ideal $\p \subset \pi_0(R)$ such that $L_{T(n)}K(R^{\sh}_{\p}) \not\simeq 0.$  
\end{lem}
\begin{proof}
Suppose the contrary, that $L_{T(n)}K(R^{\sh}_{\p})$ vanishes for every prime $\p \subset \pi_0(R).$  Since the strict henselizations of $R$ are the local rings in the \'{e}tale topology, we may write $R^{\sh}_{\p}$ as a filtered colimit $$R^{\sh}_{\p} = \colim_{i\in I} R^{(i)}_{\p},$$ where each $R^{(i)}_{\p}$ is an \'{e}tale neighborhood of $\p$.  Since $K(-)$ commutes with filtered colimits \cite[Theorem 1.1]{BGT}, we conclude that the $L_{T(n)} \colim_{i \in I}K(R^{(i)}_{\p})\simeq 0$, or equivalently that
\[
\colim_{i\in I}K(R^{(i)}_{\p}) \otimes T(n) \simeq 0.
\]
But we can choose $T(n)$ to be an $\E_1$-ring spectrum, so this is a filtered colimit of rings; thus, it vanishes if and only if it vanishes at some finite stage.  It follows that $L_{T(n)}K(R^{(i)}_{\p}) \simeq 0$ for some $i\in I$.  Thus, we have shown that every $\p \in \Spec R$ admits an \'{e}tale neighborhood on which $L_{T(n)}K(-)$ vanishes.  Since this defines an \'{e}tale sheaf by \Cref{thm:etaledesc}, we conclude that $L_{T(n)}K(R)$ must have been zero.  
\end{proof}

We are now ready to prove that $K(E_n)$ has height $n+1$.  

\begin{proof}[Proof of \Cref{thm:Eredshift}]
Choose any formal group $\mathbb{G}_0'$ over $k$ of height $n+1$, and let $E_{n+1} = E_{\mathbb{G}_0',k}$ denote the corresponding Lubin-Tate theory.  By \Cref{thm:main}, the ring $R = E_{n+1}^{tC_p}$ has the property that $L_{T(n+1)}K(R) \not \simeq 0$.  Therefore, applying \Cref{lem:stalk}, we may choose a prime ideal $\p \in \pi_0(R)$ such that $L_{T(n+1)}K(R^{\sh}_{\p}) \not\simeq 0$.  Since $L_{T(n+1)}R = 0$ (for instance, by \Cref{thm:Kuhn}), it follows from \Cref{thm:LMMT} that $L_{T(n+1)}K(L_{T(n)} R^{\sh}_{\p}) \simeq L_{T(n+1)}K(R^{\sh}_{\p})$ is also nonzero.  In fact, since $K(n)$ and $T(n)$ localization coincide for $BP$-modules (\cite[Corollary 1.10]{Hovey}), we also have $L_{T(n+1)}K(L_{K(n)} R^{\sh}_{\p}) \not\simeq 0$.  Thus, in order to see that $K(E_n)$ has height $n+1$, we will exhibit a map of $\E_{\infty}$-rings $E_n \to L_{K(n)}R^{\sh}_{\p}$.  

Let $\Gamma = \Spf (R^{\sh}_{\p})^0(\mathbb{C}P^{\infty})$ denote the canonical formal group on $\pi_0(R^{\sh}_{\p}) = \pi_0(R)^{\sh}_{\p}$. By \Cref{thm:Eunivprop}, in order to produce an $\E_{\infty}$-ring map $E_n = E_{k,\mathbb{G}_0} \to L_{K(n)}R^{\sh}_{\p}$, it suffices to exhibit a ring homomorphism $k \to \pi_0(L_{K(n)} R^{\sh}_{\p})/I_n$ and an isomorphism between $\Gamma$ and $\mathbb{G}_0$ over $\pi_0(L_{K(n)}R^{\sh}_{\p})/I_n$.  In fact, since the ideal $I_n$ is already defined in $\pi_0(R^{\sh}_{\p})$, it suffices to produce $f:k \to \pi_0(R^{\sh}_{\p})/I_n$ and an isomorphism $f^*\mathbb{G}_0 \cong \Gamma_{\pi_0(R^{\sh}_{\p})/I_n}$ of formal groups.  

Since $\pi_0(R) = \pi_0(E_{n+1}^{tC_p})$ is naturally an algebra over the Witt vectors $W(k)$ and $p\in I_n$, there is a canonical inclusion $k\subset \pi_0(R)/I_n$; composing this with the map $\pi_0(R)/I_n \to \pi_0(R^{\sh}_{\p})/I_n$, we define the ring homomorphism $f:k \to \pi_0(R^{\sh}_{\p})/I_n.$  It remains to show the requisite isomorphism of formal groups.  

\begin{lem}
Let $\kappa(\p) = \pi_0(R^{\sh}_{\p}) / \p$ denote the residue field at $\p$ and let $\Gamma_{\kappa(\p)}$ denote the reduction of the formal group $\Gamma$ modulo $\p$.  Then $\Gamma_{\kappa(\p)}$ has height $n$.
\end{lem}
\begin{proof}[Proof of Lemma]
$\Gamma_{\kappa(\p)}$ has height at most $n$ because the ring $L_{K(n+1)}R= L_{K(n+1)} E_{n+1}^{tC_p} = 0$ (for instance by \Cref{thm:Kuhn}).  If $\Gamma_{\kappa(\p)}$ had height at most $n-1$, then some $v_i$, $0\leq i \leq n-1$, would be invertible in $\kappa(\p)$.  Since $\pi_0(R)^{\sh}_{\p}$ is a local ring, this means that this $v_i$ is invertible in $\pi_0(R^{\sh}_{\p})$, which implies that $R^{\sh}_{\p}$ would be $T(n)\vee T(n+1)$-acyclic.  By \Cref{thm:LMMT}, this would imply $L_{T(n+1)}K(R^{\sh}_{\p}) \simeq 0$ which contradicts our choice of $\p$. 
\end{proof}

It follows that the ideal $I_n \subset \pi_0(R^{\sh}_{\p})$ is contained in $\p$.  Consequently, the ring $\pi_0(R^{\sh}_{\p})/I_n$ is strictly henselian (with maximal ideal $\p$), and $\Gamma_{\pi_0(R^{\sh}_{\p})/I_n}$ is a formal group of height exactly $n$ over it.  Recall that any two formal group laws of height $n$ over a strictly henselian ring are isomorphic (this is essentially a result of Lazard; see, for instance, \cite[Lecture 13, Theorem 11]{LurChrom}).  In particular, $\Gamma_{\pi_0(R^{\sh}_{\p})/I_n}$ and $f^*\mathbb{G}_0$ are isomorphic, as desired.  
\end{proof}

\section{Iterated $K$-theory}\label{sec:iteratedK}

\begin{ntn*}
For a functor $G$, we will use the superscript $G^{(n)}$ to denote the $n$th iterate of $G$.  
\end{ntn*}

In this section, we prove \Cref{thm:iteratedK}, which we restate here for the reader's convenience:

\begin{thm:iteratedK}
Let $n\geq 0$ and let $K^{(n)}$ denote the $n$-fold iterate of algebraic $K$-theory.  Then, for any field $k$ of characteristic different from $p$, the spectrum $L_{T(n)}K^{(n)}(k)$ is nonzero.  
\end{thm:iteratedK}

For the purposes of exposition, we will start by treating the case $k=\Q$ in \Cref{sect:Q} and prove the general case in the following sections.  The idea for extending to the general case and the arguments in \Cref{sub:generalk} are due to Akhil Mathew: we are grateful to him for allowing us to include his work in this paper; however, any errors are solely the responsibility of the author.

\subsection{The case $k=\Q$}\label{sect:Q}

\begin{ntn*}
Let $\tp : \Sp \to \Sp$ be the functor given by the formula $$\tp X = (\tau_{\geq 0}X)^{tC_p}.$$  As usual, the operation $(-)^{tC_p}$ is taken with respect to the trivial action of $C_p$.  
\end{ntn*}

Let $E_n$ be a height $n$ Lubin-Tate theory and set $R := \tp^{(n)}E_n$.  The nonvanishing of $L_{T(n)}K^{(n)}(\Q)$ follows from the following two propositions about $R$: 

\begin{prop}\label{prop:iterated1}
There exists a canonical $\E_{\infty}$-ring map $L_{T(n)}K^{(n)}(\Q) \to  L_{T(n)}K^{(n)}(R)$.  
\end{prop}

\begin{prop}\label{prop:iterated2}
The iterated algebraic $K$-theory spectrum $K^{(n)}(R)$ is not $T(n)$-acyclic. 
\end{prop}

The proof of \Cref{prop:iterated1} relies on the theorem of \cite{LMMT} (\Cref{thm:LMMT}), which implies that one can replace $R$ by its localization at $T(0)\vee T(1)\vee \cdots \vee T(n)$, which we denote by $L_n^{p,f} R$.   Essentially by the blueshift phenomenon for Tate cohomology, this localization $L_n^{p,f} R$ is rational, and thus receives a ring map from $\Q$.  The proof of \Cref{prop:iterated2} uses the iterated Dennis trace map and an iterated variant of \Cref{prop:tS1height}.  We now give the full proofs:

\begin{proof}[Proof of \Cref{prop:iterated1}]

First, observe that by the theorem of Land-Mathew-Meier-Tamme (\Cref{thm:LMMT}), the map $$L_{T(n)}K^{(n)}(R) \to L_{T(n)}K(L_{T(n)\vee T(n-1)} K^{(n-1)}(R))$$ is an equivalence.  Applying \Cref{thm:LMMT} in this manner repeatedly, we conclude that the natural map $$L_{T(n)}K^{(n)}(R) \to L_{T(n)}K^{(n)}(L_{n}^{p,f} R)$$ is an equivalence.  Thus, it suffices to produce a map $\Q \to L_{n}^{p,f} R$.  This follows immediately from the following lemma, which implies that $L_{n}^{p,f} R \simeq L_{T(0)} R$, and is therefore rational (as everything in sight is $p$-local):


\begin{lem}
The ring $R$ is $T(1)\vee T(2)\vee \cdots \vee T(n)$-acyclic.  
\end{lem}
\begin{proof}
This lemma can be approached computationally, but we will give a more hands-off proof using genuine equivariant homotopy theory; the reader is referred to \cite{LMS, MNN} for additional background.  Let $G = (C_p)^{\times n}$, and regard $E_n$ as a genuine $G$-equivariant spectrum such that it is Borel-equivariant (i.e., the natural map $E_n^H\to E_n^{hH}$ is an equivalence for all $H\subset G$), and the underlying $G$-action is trivial.  Let $e_n$ denote the equivariant connective cover of the genuine equivariant spectrum $E_n$ (in the usual Mackey $t$-structure), which has the feature that the natural map $e_n \to E_n$ induces connective covers
$$e_n^H \simeq  \tau_{\geq 0}(E_n^{H}) \to E_n^H$$  
on genuine fixed points for all $H\subset G$.

\begin{rmk}\label{rmk:geometric}
For a genuine $G$-spectrum $X$, recall that one can consider its geometric fixed point spectrum $\Phi^G X$, which admits the following description (for any finite abelian group $G$): it fits into a cofiber sequence $$Y \to X^{G} \to \Phi^{G} X$$ where the spectrum $Y$ is given by the formula $$Y \simeq \colim_{H \in \mathrm{Sub}(G)^-} (X^{H})_{hG/H}.$$  Here, $\mathrm{Sub}(G)^-$ denotes the poset of proper subgroups of $G$ ordered by inclusion, and the maps in the diagram are the genuine additive transfer maps; accordingly, the map $Y\to X^G$ is induced by the genuine additive transfer.  
\end{rmk}

We first note the following variant of \Cref{lem:covertate}:

\begin{lem}\label{lem:covertategen}
The natural map $L_{T(j)}\Phi^{G} e_n  \to L_{T(j)} \Phi^G E_n$ is an equivalence for $j\geq 1$.
\end{lem}
\begin{proof}
By \Cref{rmk:geometric}, it suffices to check that the natural map $$L_{T(j)}(e_n^H)_{hG/H} \to L_{T(j)}(E_n^H)_{hG/H}$$ is an equivalence for all $H\subset G$.  But the cofiber of the map $e_n^H \to E_n^H$ is the coconnective spectrum $\tau_{<0} (E_n^{hH})$.   Since $L_{T(j)}$ and homotopy orbits commute with colimits, the conclusion follows from the fact that $T(j)$-localization annihilates coconnective spectra.
\end{proof}

On the other hand, we claim that there is a map of $\E_{\infty}$-rings $\Phi^G e_n \to R = \tp^{(n)}E_n.$ This is because for each $1\leq k \leq n$, there is a natural map $$\Phi^{C_p^{\times k}} e_n \simeq \Phi^{C_p} \Phi^{C_p^{\times (k-1)}} e_n \to (\Phi^{C_p^{\times (k-1)}}e_n)^{tC_p}.$$  Since $\Phi^{C_p^{\times k}}e_n$ is connective (for instance, by the above description), the map factors through the connective cover $\tau_{\geq 0}(\Phi^{C_p^{\times (k-1)}}e_n)^{tC_p}$.  Composing these for $1\leq k \leq n$, we obtain the desired map.

Combining this with \Cref{lem:covertategen}, we obtain an $\E_{\infty}$-ring map $$L_{T(j)} \Phi^G E_n \to L_{T(j)} \tp^{(n)}E_n = L_{T(j)}R.$$  But by \cite[Theorem 3.5]{BHNNNS}, which computes the blueshift for geometric fixed points of finite abelian groups on Lubin-Tate theories, the spectrum $\Phi^G E_n$ is rational (this also goes back to \cite{HKR}, cf. \Cref{sect:pthroot}).  In particular, for any $1\leq j \leq n$, we have $L_{T(j)}\Phi^G E_n \simeq 0$, which implies that $L_{T(j)}R \simeq 0$ as well.  
\end{proof}

\end{proof}

We now prove \Cref{prop:iterated2} (and with it, the $k=\Q$ case of \Cref{thm:iteratedK}) -- the proof is a relatively straightforward adaptation of the proof of \Cref{thm:main}, so we will assume some familiarity with that proof.

\begin{proof}[Proof of \Cref{prop:iterated2}]
By iterating the Dennis trace map, one obtains a ring map
$$K^{(n)}(R) \to \THH^{(n)}(R)$$
which is $(S^1)^{\times n}$-equivariant for the trivial action on the source and the natural action on the target.  It therefore induces a map of rings $$K^{(n)} (R) \to \THH^{(n)}(R)^{h(S^1)^{\times n}}.$$

\begin{rmk}\label{rmk:iteratednotation}
We have seen that the spectrum $\tp X = (\tau_{\geq 0}X)^{tC_p}$ acquires a residual action of $S^1/C_p$.  Accordingly, $\tp^{(k)} X$ acquires a canonical action of $(S^1/C_p)^{\times k}$; to distinguish the iterates, we will label them by writing $$\tp^{(k)}X = (\tau_{\geq 0}(\cdots (\tau_{\geq 0} X)^{tC_p^{(1)}} \cdots ))^{tC_p^{(k)}},$$ which admits an action of $S^1/C_p^{(1)} \times \cdots \times S^1/C_p^{(k)} \cong (S^1)^{\times k}$.  
\end{rmk}

Using (the $(S^1)^{\times n}$ analogue of) \Cref{THHuprop} as in the proof of \Cref{thm:main}, we see that the identity map of $R = \tp^{(n)}E_n$ determines an $\E_{\infty}$-ring map $$\THH^{(n)}(R)^{h(S^1)^{\times n}} \to R^{h(S^1/C_p)^{\times n}}.$$  Thus, to show that $L_{T(n)}K^{(n)}(R) \not\simeq 0$, it suffices (by \Cref{rmk:ringmap}), to show that: 

\begin{lem}
The ring $R^{h(S^1/C_p)^{\times n}}$ has nontrivial $T(n)$-localization.  
\end{lem}
\begin{proof}
In fact, we will show by induction that for any $k\geq 0$, the ring $(\tp^{(k)}E_n)^{h(S^1/C_p)^{\times k}}$  has nontrivial $T(n)$-localization.   The base case $k=0$ is immediate.  For the inductive step $k\geq 1$, referring to the notation in \Cref{rmk:iteratednotation}, we write
\[
\tp^{(k)}E_n = \tp(\tp^{(k-1)}E_n) = (\tau_{\geq 0} \tp^{(k-1)}E_n)^{tC_p^{(k)}}.
\]
Then, applying the Tate orbit lemma (\Cref{prop:tateorbit}), we have
\[
(\tp^{(k)}E_n)^{hS^1/C_p^{(k)}} =  ((\tau_{\geq 0} \tp^{(k-1)}E_n)^{tC_p^{(k)}})^{hS^1/{C_p^{(k)}}} \simeq_p (\tau_{\geq 0}\tp^{(k-1)}E_n)^{tS^1},
\]
where $\simeq_p$ denotes equivalence after $p$-completion.  Since $\tau_{\geq 0}\tp^{(k-1)}E_n$ is complex orientable (as it receives a ring map from the complex orientable ring $\tau_{\geq 0}E_n$), it is a retract of its own $S^1$-Tate construction.  It follows that 
\[
(\tp^{(k)}E_n)^{h(S^1/C_p)^{\times k}} \simeq ((\tp^{(k)}E_n)^{hS^1/C_p^{(k)}})^{h(S^1/C_p)^{\times k-1}} 
\]
has $(\tau_{\geq 0}\tp^{(k-1)}E_n)^{h(S^1/C_p)^{\times k-1}}$ as a retract.  By the inductive hypothesis, this latter spectrum has nontrivial $T(n)$-localization, which completes the proof.

\end{proof}

\end{proof}

\subsection{Adding a primitive $p$th root of unity}\label{sect:pthroot}

Here, we make an additional observation about the proof of \Cref{prop:iterated1} which allows us to extend the results of the previous section to $k=\Q(\zeta_p)$.  Recall that $E_n$ denotes a Lubin-Tate theory of height $n$ and residue field $k$, equipped with a genuine $C_p^{\times n}$ equivariant structure via the Borel equivariant trivial action, and $e_n$ denotes the equivariant connective cover of $E_n$.    Our goal will be to prove the following:

\begin{prop}\label{prop:pthroot}
Suppose that the residue field $k$ of $E_n$ is algebraically closed.  Then the ring $ \pi_0(\Phi^{C_p^{\times n}}e_n)\otimes \Q$ contains a primitive $p$th root of unity.  
\end{prop}

We first recall some computational facts:

\begin{rmk}\label{rmk:cohomology}

Fixing a complex orientation on $E_n$, we have an isomorphism (cf. \cite[\S 5.4]{HKR}) 
\[
\pi_* E_n^{hC_p^{\times n}} \cong (E_n)_*[[\tilde{x}_1, \tilde{x}_2, \cdots \tilde{x}_n]]/([p](\tilde{x}_1), [p](\tilde{x}_2), \cdots , [p](\tilde{x}_n)),
\]
where $\tilde{x}_i \in E_n^{2}(BC_p^{\times n})$ is the first Chern class of the line bundle corresponding to the $i$th projection $C_p^{\times n} \to C_p$, and $[p]$ denotes the $p$-series of (the chosen coordinate on) the formal group $\Spf E_n^*(\mathbb{C}P^{\infty})$ over $\pi_*(E_n)$.  Choosing a periodicity element $u\in \pi_2E_n$ and setting $x_i = u \tilde{x}_i$, we obtain an isomorphism
\begin{equation}\label{eqn:pi0EhG}
\pi_0E_n^{hC_p^{\times n}} \simeq \pi_0(E_n)[[x_1, x_2, \cdots x_n]]/([p]_0(x_1), [p]_0(x_2), \cdots , [p]_0(x_n)),
\end{equation}
where $[p]_0$ denotes the $p$-series shifted into degree zero, so that $[p]_0(x_i) = u[p](\tilde{x}_i)$.  In other words, $[p]_0$ is the $p$-series of the formal group $\Spf E_n^0(\mathbb{C}P^{\infty})$ over $\pi_0(E_n)$, which we will denote by $\G_{E_n}$.  Note that, essentially by definition of equivariant connective cover, there is an isomorphism of rings $\pi_0 e_n^{C_p^{\times n}} \cong \pi_0 E_n^{hC_p^{\times n}}$.

By \cite[Proposition 3.20]{GM}, the ring $\pi_0 \Phi^{C_p^{\times n}} E_n$ is obtained from the ring $\pi_0 E_n^{hC_p^{\times n}}$  by inverting the Chern classes of the line bundles corresponding to the nontrivial characters of $C_p^{\times n}$.  These are the elements 
\[
[a_1]_0(x_1) +_{\G_{E_n}} [a_2]_0(x_2) +_{\G_{E_n}} \cdots +_{\G_{E_n}} [a_n]_0(x_n)
\]
for $0\leq a_i \leq p-1$.
\end{rmk}

We then have the following variant of \Cref{prop:pthroot}:

\begin{prop}\label{prop:pthroothelper}
Suppose that the residue field $k$ of $E_n$ is algebraically closed.  Then the ring $\pi_0 \Phi^{C_p^{\times n}} E_n$ contains a primitive $p$th root of unity.  
\end{prop}
\begin{proof}
By \cite[Proposition 6.2]{HKR} and the description of $\pi_0 \Phi^{C_p^{\times n}}E_n$ above, there is an isomorphism of finite flat group schemes
\begin{equation}\label{eqn:gpschemes}
(\underline{C}_p^{\times n})_{\pi_0 \Phi^{C_p^{\times n}} E_n} \simeq (\G_{E_n}[p])_{\pi_0 \Phi^{C_p^{\times n}} E_n}
\end{equation}
between the constant group scheme at $C_p^{\times n}$ and the $p$-torsion scheme of the formal group $\G_{E_n}$
(in fact, $\pi_0 \Phi^{C_p^{\times n}} E_n$ is the universal $\pi_0E_n$-algebra with such an isomorphism).  

\begin{rmk}
Let $R$ be a commutative ring and let $H$ be a finite flat commutative group scheme over $R$.  Then Hopkins and Lurie show that, for $d\geq 0$, one can associate to $H$ its \emph{group scheme of alternating maps} $\Alt{d}{H}$ \cite[Construction 3.2.11]{HL}, which can be thought of as classifying alternating maps $H^{\times d} \to \G_m$ (cf. \cite[\S 3]{HL} for details).  In the special case that $H = \Gamma[p^t]$ is the $p^t$-torsion in a $p$-divisible group $\Gamma$ of height $n$ and dimension $1$, the schemes $\Alt{d}{\Gamma[p^t]}$ are also finite flat commutative group schemes, and their Cartier duals (denoted $\bD$) form a directed system
\[
\bD(\Alt{d}{\Gamma[1]}) \to \bD(\Alt{d}{\Gamma[p]}) \to \bD(\Alt{d}{\Gamma[p^2]}) \to \cdots
\]
which defines a $p$-divisible group of height $\binom{n}{d}$ and dimension $\binom{n-1}{d-1}$ \cite[Corollary 3.5.5]{HL}.  We will denote this $p$-divisible group by $\wedge^d \Gamma$ and refer to it as the \emph{$d$th exterior power} of $\Gamma$.  

We record two additional features of this construction:
\begin{enumerate}
\item For a $p$-divisible group $\Gamma$ of height $n$ and dimension $1$, the ``top'' exterior power $\wedge^n \Gamma$ has height $1$ and dimension $1$.  In the special case where $\Gamma = \G_{E_n}[p^{\infty}]$ is the $p$-divisible group corresponding to the  Lubin-Tate theory $E_n$ of height $n$ with algebraically closed residue field $k$, there is an isomorphism of group schemes $\wedge^n \G_{E_n}[p] \cong \mu_p$ over $\pi_0E_n$ \cite[Proposition 5.3.30]{HL}.
\item For the constant group scheme $\underline{C}_p^{\times n}$ over a commutative ring $R$, there are isomorphisms
\[
\wedge^n \underline{C}_p^{\times n} := \bD(\Alt{n}{ \underline{C}_p^{\times n}}) \cong \bD(\mu_p) \cong \underline{C}_p
\]
by \cite[Remark 3.2.21]{HL}.  
\end{enumerate}
\end{rmk}

Thus, applying $\wedge^n$ to the isomorphism (\ref{eqn:gpschemes}) of group schemes and using the above remark, we find that the finite flat group schemes $\wedge^n  \underline{C}_p^{\times n}\cong \underline{C}_p$ and $\wedge^n \G_{E_n}[p] \cong \mu_p$ are isomorphic over $\pi_0 \Phi^{C_p^{\times n}} E_n$.  It follows that the ring $\pi_0 \Phi^{C_p^{\times n}} E_n$ contains a primitive $p$th root of unity.  
\end{proof}

We are now ready to prove \Cref{prop:pthroot}.

\begin{proof}[Proof of \Cref{prop:pthroot}]
By \Cref{prop:pthroothelper}, $\pi_0 \Phi^{C_p^{\times n}}E_n$ contains a primitive $p$th root of unity.   Hence, it will be enough to exhibit a ring homomorphism fitting into the square:\footnote{In fact, the map we construct turns out to be an equivalence, but we will not need this.}
\begin{equation}\label{eqn:pthrootsquare}
\begin{tikzcd}
\pi_0 E_n^{hC_p^{\times n}} \arrow[d] \arrow[r, "\cong"]& \pi_0 e_n^{C_p^{\times n}}  \arrow[d]\\
\pi_0 \Phi^{C_p^{\times n}}E_n \arrow[r, dashed] & \pi_0(\Phi^{C_p^{\times n}}e_n) \otimes \Q.
\end{tikzcd}
\end{equation}
Here, the vertical maps are the canonical maps from genuine fixed points to geometric fixed points (composed with rationalization on the right).  Recalling the descriptions of these groups given in \Cref{rmk:cohomology}, the left vertical map is given by inverting elements $\sum_{\G_{E_n}}[a_i]_0(x_i)$, where $0\leq a_i\leq p-1$.  Thus, to produce the desired factorization, it suffices to see that these elements are invertible in $\pi_0(\Phi^{C_p^{\times n}}e_n) \otimes \Q$.  In fact, by symmetry, it suffices to see that the image of $x_1$ in $\pi_0(\Phi^{C_p^{\times n}}e_n) \otimes \Q$ is invertible.  

Note that by definition of geometric fixed points, the map $e_n^{C_p^{\times n}} \to \Phi^{C_p^{\times n}}e_n$ is nullhomotopic when precomposed with additive transfers from any proper subgroup of $C_p^{\times n}$: in particular, if we let $\mathrm{tr}^{(1)}$ denote the composite
\[
\mathrm{tr}^{(1)}: e \xrightarrow{\mathrm{tr}} e^{C_p} \xrightarrow{\mathrm{proj}_1^*} e^{C_p^{\times n}}
\]
(where $\mathrm{proj}_1^*$ denotes the map induced by projection $C_p^{\times n}\to C_p$ onto the first component), then the image of $\mathrm{tr}^{(1)}$ in $\pi_0 e_n^{C_p^{\times n}}$ is in the kernel of the map $\pi_0 e_n^{C_p^{\times n}} \to \pi_0 \Phi^{C_p^{\times n}}e_n$.  By \cite[Remark 6.15]{HKR}, we have the formula 
\[
\mathrm{tr}^{(1)}(1) = [p](\tilde{x}_1)/\tilde{x}_1 = [p]_0(x_1)/x_1,
\]
where $1\in \pi_0 e_n$ denotes the unit.  Thus, $[p]_0(x_1)/x_1 = \mathrm{tr}^{(1)}(1)$ vanishes in $\pi_0 (\Phi^{C_p^{\times n}}e_n)\otimes \Q$.  But the series $[p]_0(x_1)/x_1$ has the form $p+x_1g(x_1)$ for some series $g \in (e_n)_0[[x_1]]$ and $p$ is invertible in $\pi_0 (\Phi^{C_p^{\times n}}e_n)\otimes \Q$, from which it follows that $x_1$ is also invertible, as desired.  
\end{proof}

\subsection{The case of an arbitrary field $k$}\label{sub:generalk}

Following arguments of Akhil Mathew, we now extend the results of \Cref{sect:Q} to prove \Cref{thm:iteratedK} for all fields $k$.  The proof rests on the work of Clausen-Mathew-Naumann-Noel \cite{CMNN} on Galois descent for telescopically localized $K$-theory and work of Barthel-Carmeli-Schlank-Yanovski \cite{BCSY} on $T(n)$-local homotopy theory.  We refer the reader to \cite{RognesGal} for background on the notion of ($T(n)$-local) $G$-Galois extensions of ring spectra in the sense of Rognes.

\begin{thm}\label{thm:pgroupdescent}
Let $G$ be a finite $p$-group and let $R\to R'$ be a map of $\E_{\infty}$-algebras over $\Q$ which is a $G$-Galois extension in the sense of \cite{RognesGal}.  Then for $n\geq 0$, the natural map
\[
L_{T(n)}K^{(n)}(R) \to L_{T(n)}K^{(n)}(R')
\]
is a $T(n)$-local $G$-Galois extension.  In particular, it induces an equivalence $L_{T(n)}K^{(n)}(R) \to L_{T(n)}K^{(n)}(R')^{hG}$.  
\end{thm}
\begin{proof}
We proceed by induction.  The case $n=0$ is true by assumption, so suppose $n\geq 1$ and assume that the map 
\[
L_{T(n-1)} K^{(n-1)}(R) \to L_{T(n-1)}K^{(n-1)}(R')
\]
is a $T(n-1)$-local $G$-Galois extension.  Then, by \cite[Corollary 4.13]{CMNN}, the induced map
\begin{equation}\label{eqn:gal1}
L_{T(n)} K( L_{T(n-1)} K^{(n-1)}(R) ) \to (L_{T(n)}K( L_{T(n-1)}K^{(n-1)}(R')))^{hG}
\end{equation}
is an equivalence.  On the other hand, by \cite{CMNN} (\Cref{thm:introCMNN}), the $T(n)$-localization of $K^{(n-1)}(A)$ vanishes for any rational $\E_{\infty}$-algebra $A$, and hence by \cite{LMMT} (\Cref{thm:LMMT}), there is a natural equivalence
\[
L_{T(n)}K^{(n)}(A) \simeq L_{T(n)}K( L_{T(n)\vee T(n-1)} K^{(n-1)}(A)) = L_{T(n)} K(L_{T(n-1)}K^{(n-1)}(A)).
\]
Combining this with (\ref{eqn:gal1}), it follows that the natural map
\[
L_{T(n)}K^{(n)}(R) \to (L_{T(n)}K^{(n)}(R'))^{hG}
\]
is an equivalence.  The inductive step then follows from the following theorem of \cite{BCSY}.
\end{proof}

\begin{thm}[Barthel-Carmeli-Schlank-Yanovski]
Let $G$ be a finite $p$-group and let $E$ be a $T(n)$-local $\E_{\infty}$-ring spectrum with an action of $G$.  Then the natural map $E^{hG}\to E$ is a $T(n)$-local $G$-Galois extension.  
\end{thm}
This theorem will be a consequence of the results in the forthcoming work \cite{BCSY}, but we indicate some of the proof for the reader's convenience:
\begin{proof}[Proof outline]
It suffices to check that the natural map
\[
L_{T(n)}(E\otimes_{E^{hG}} E) \to \prod_G E
\]
(induced by the $G$-action and multiplication) is an equivalence.  But the map is an equivalence after applying the functor $(-)^{hG}$: this is because by Kuhn's blueshift theorem (\Cref{thm:Kuhn}), $T(n)$-local orbits and fixed points coincide, so 
\begin{align*}
(L_{T(n)}(E\otimes_{E^{hG}} E))^{hG} \simeq L_{T(n)} (E\otimes_{E^{hG}}E)_{hG} &\simeq L_{T(n)}(L_{T(n)}E_{hG}\otimes_{E^{hG}} E) \\
&\simeq L_{T(n)} (E^{hG}\otimes_{E^{hG}} E) \simeq E.
\end{align*}
The conclusion then follows from the fact that $(-)^{hG}$ is conservative on $T(n)$-local spectra -- this statement, whose $K(n)$-local analogue is \cite[Corollary 5.4.4]{HL}, will appear in \cite{BCSY}.
\end{proof}

We are now ready to complete the proof of \Cref{thm:iteratedK}.  

\begin{proof}[Proof of \Cref{thm:iteratedK}]
We make a series of reductions:
\begin{enumerate}
\item By mapping to an algebraic closure, it suffices to demonstrate the statement for an algebraically closed field $k$ (of characteristic different from $p$).  
\item In this case, by work of Suslin \cite{Suslin1, Suslin2}, one has that $L_{T(1)}K(k) \simeq KU^{\wedge}_p$.  On the other hand, by \cite[Theorem 1.4]{BCM}, there is an equivalence of $\E_{\infty}$-rings
\[
L_{T(1)}K(\Q(\zeta_{p^{\infty}})) \simeq L_{T(1)}(KU^{\wedge}_p \otimes K(\Q)).
\]
Since this receives a ring map from $KU^{\wedge}_p$,  we are reduced to the case $k = \Q(\zeta_{p^{\infty}})$. 
\item In fact, because $K$-theory commutes with filtered colimits, we have that
\[
T(n) \otimes K^{(n)}(\Q(\zeta_{p^{\infty}})) \simeq  \colim_j \big( T(n) \otimes K^{(n)}(\Q(\zeta_{p^j})) \big).
\]
We may choose $T(n)$ to be an $\E_1$-ring spectrum, making this colimit a filtered colimit of rings.  
As a filtered colimit of rings vanishes if and only if it does at a finite stage, we are reduced to showing that $L_{T(n)}K^{(n)}(\Q(\zeta_{p^j})) \not\simeq 0$ (as $j$ ranges over the natural numbers).
\item The map of rings $\Q(\zeta_p) \to \Q(\zeta_{p^j})$ is a Galois extension with Galois group $C_{p^{j-1}}$, and therefore by \Cref{thm:pgroupdescent}, we have an equivalence 
\[
L_{T(n)}K^{(n)}(\Q(\zeta_{p}))\simeq L_{T(n)}K^{(n)}(\Q(\zeta_{p^j}))^{hC_{p^{j-1}}}.
\]
  Therefore, we are reduced to the case of the field $k=\Q(\zeta_p)$.  
\end{enumerate}

The case of $\Q(\zeta_p)$ follows by using the results of \Cref{sect:pthroot} to strengthen the proof of \Cref{prop:iterated1}.  Namely,  the ring $R$ (from the proof of \Cref{prop:iterated1}) receives a ring map from $\Phi^{C_p^{\times n}}e_n$, and therefore by \Cref{prop:pthroot}, $\pi_0(R\otimes \Q)$ contains a primitive $p$th root of unity.  Since the map $\Q\to \Q(\zeta_p)$ is \'{e}tale, the map $\Q \to L_{T(0)}R \simeq R\otimes \Q$ in the proof of \Cref{prop:iterated1} extends to a ring map $\Q(\zeta_p)\to L_{T(0)}R $ and one obtains an $\E_{\infty}$-ring map 
\[
L_{T(n)}K^{(n)}(\Q(\zeta_p)) \to L_{T(n)}K^{(n)}(R).
\]
The rest of the argument in \Cref{sect:Q} shows that the target of this map is nonzero, so $L_{T(n)}K^{(n)}(\Q(\zeta_p)) \not\simeq 0$, and we are done.  
\end{proof}

\bibliographystyle{alpha}
\bibliography{Bibliography}

\end{document}